\documentclass[12pt]{amsart}

\usepackage[all,cmtip]{xy}
\usepackage{graphicx}
\usepackage{amsmath}
\usepackage{amsfonts}
\usepackage{amssymb}
\usepackage{amsthm}
\usepackage{amscd}
\usepackage{textcomp}
\usepackage{hyperref}
\usepackage{extarrows}

\DeclareFontEncoding{OT2}{}{} 
  \newcommand{\textcyr}[1]{%
    {\fontencoding{OT2}\fontfamily{wncyr}\fontseries{m}\fontshape{n}%
     \selectfont #1}}
\newcommand{\Sha}{{\mbox{\textcyr{Sh}}}}
\newcommand{\Cok}{{\mbox{\textcyr{B}}}}

\newcommand{\hra}{\hookrightarrow}

\newcommand{\Q}{\mathbb{Q}}

\newcommand{\A}{{\mathbb A}}

\newcommand{\G}{{\mathbb G}}
\newcommand{\I}{{\mathbb I}}

\newcommand{\Z}{\mathbb{Z}}

\newcommand{\isom}{\cong}

\DeclareMathOperator{\coker}{coker}
\DeclareMathOperator{\Br}{Br}

\DeclareMathOperator{\Gal}{Gal}

\DeclareMathOperator{\Nm}{Nm}

\DeclareMathOperator{\Pic}{Pic}

\DeclareMathOperator{\inv}{inv}

\DeclareMathOperator{\Spec}{Spec}

\DeclareMathOperator{\Hom}{Hom}

\theoremstyle{plain}
\newtheorem{theorem}{Theorem}[section]
\newtheorem{proposition}[theorem]{Proposition}
\newtheorem{corollary}[theorem]{Corollary}

\newtheorem{lemma}[theorem]{Lemma}

\theoremstyle{definition}

\theoremstyle{remark}

\newtheorem{remark}[theorem]{Remark}

\numberwithin{equation}{section}

 1

\newcommand{\thetitle}
{Brauer groups of torsors under algebraic tori}

\begin{document}

\title{\thetitle}
\author{Saikat Biswas}
\address{School of Mathematical and Statistical Sciences, Arizona State University,
Tempe, AZ}
\email{Saikat.Biswas@asu.edu}

\begin{abstract}
Let $T$ be an algebraic torus defined over a global field $K$. For any $K$-torsor $X$ under $T$, we relate the Brauer group of $X$ to the ad\'{e}le class group of $T$ as well as to the Shafarevich Tate group of $T$.
\end{abstract}

\subjclass[2010]{Primary 11E72; Secondary 11R37}

\keywords{Algebraic tori, ad\`{e}le class group, Brauer groups, Shafarevich-Tate groups, global fields.}

\maketitle


\section{Introduction}
Let $L/K$ be a finite extension of global fields with Galois group $G$. Let $T$ be an algebraic torus of dimension $d$ defined over $K$ and split over $L$, so that $T_L\isom\G_m^d$. Let $C_L$ be the idele class group of $L$ and let $C_L(T)$ denote the \emph{ad\'{e}le class group} of $T$ over $L$. There is a canonical action of $G$ on $C_L(T)$, and the corresponding Tate cohomology groups $\hat{H}^i(G,C_L(T))$ are finite. Consider a smooth, geometrically integral $K$-variety $X$ which is a torsor under $T$. Let $\Br(X)=H^2(X,\G_m)$ be the \emph{cohomological Brauer group} of $X$. The kernel of the natural map $\Br(X)\to\Br(\overline{X})$, where $\overline{X}=X\otimes_{K}\overline{K}$,  is called the \emph{algebraic Brauer group} of $X$, and denoted by $\Br_1(X)$. There is a map $\Br_1(X)\to\Br_1(X_L)^G$  whose kernel we denote by $\Br_1(X_{L/K})$. As we show in this paper, there is also a map $\Br(L/K)\to\Br_1(X_{L/K})$ whose image we denote by $\Br_0(X_{L/K})$. Under this setting, we show that a consequence of one of the main results of this paper (see Theorem \ref{t6}) is

\begin{theorem}\label{mt1}
The quotient group $\Br_1(X_{L/K})/\Br_0(X_{L/K})$ is finite, and its order divides that of ${C_L(T)}^{G}/\Nm{C_L(T)}$. The orders are equal when $L/K$ is a finite, cyclic extension.
\end{theorem}

Let $v$ be any prime of $K$ and $w$ be that of $L$ dividing $v$. Let $K_v$ and $L_w$ denote the corresponding completions.
We define the group $\Br_1(X_{L_w/K_v})$ similar to the global case above. Let $\Br'_1(X_{L/K})$ and $\Cok$ be defined by the exactness of the sequence
$$1 \to \Br'_1(X_{L/K}) \to \Br_1(X_{L/K}) \to \bigoplus_{v}\Br_1(X_{L_w/K_v}) \to \Cok \to 1$$
where the sum is over all primes $v$ of $K$. The \emph{Shafarevich-Tate group} of $T$, denoted by $\Sha(T/K)$, represents the group of isomorphism classes of principal homogeneous spaces of $T$ that have a $K_v$-rational point for all $v$ but does not have any non-trivial $K$-rational points. It is known that $\Sha(T/K)$ is finite. The second main result in this paper relates $ \Br'_1(X_{L/K})$ to $\Sha(T/K)$.

\begin{theorem}\label{mt2}
Suppose that $X(K_v)\neq\emptyset$ for all primes $v$. Then $\Br'_1(X_{L/K})$ is finite and its order divides that of $\Sha(T/K)$. Furthermore, if $L/K$ is cyclic then we have
$$[\Br'_1(X_{L/K})]=[{C_L(T)}^{G}\cap\Nm(C_L):C_K(T)]$$
and
$$\frac{[\Br'_1(X_{L/K})]}{[\Cok]}=\frac{[\Sha(T/K)]}{[L:K][T(\A_K)\cap\Nm(C_L(T)):\Nm(T(\A_L))]}$$
where $\Nm$ is the norm map on $L/K$, and $T(\A_L)$ is the adele group of $T$ over $L$.
\end{theorem}

We now briefly describe the organization of this paper. In Section 2, we summarize some basic results regarding the arithmetic of algebraic tori. In Section 3, we briefly discuss some results from global class field theory. In Section 4, we study the Picard group and the Brauer group of torsors under tori. In Section 5, we state and prove the main results of this paper.

\section{Galois cohomology of algebraic tori}
Let $T$ be an algebraic torus defined over any perfect field $K$. Thus $T$ is an algebraic group for which there is an isomorphism $T\isom\G_m^d$, where the integer $d$ is the dimension of $T$. In general, this isomorphism is defined over some finite field extension $L$ of $K$, called the \emph{splitting field} of $T$. Let $X^{*}(T):=\Hom(T,\G_m)$ be the group of \emph{characters} of $T$. For a $d$-dimensional torus $T$, it follows that $X^{*}(T)\isom\Z^d$. In particular, $X^{*}(T)$ is a finitely-generated torsion-free $\Z$-module. To say that an extension $L$ of $K$ is a splitting field of $T$ is equivalent to saying that $X^{*}(T)=X^{*}(T)_L$, i.e. all its characters are defined over $L$. This can also be explained in terms of the natural continuous action of the profinite group $G_K:=\Gal(\overline{K}/K)$ on the discrete group $X^{*}(T)$, which gives the latter the structure of a $\Z[G_K]$-module. In such a case, when $L$ is the splitting field of $T$, the open subgroup $G_L\subset G_K$ corresponding to $L$ acts trivially on $X^{*}(T)$, i.e. $X^{*}(T)={X^{*}(T)}^{G_L}$. Note further that we also have 
$$T(L)\isom L^{\times{d}}\isom\Hom(\Z^d,L^{\times})=\Hom(X^{*}(T),L^{\times})$$
so that we may also write $T=\Hom(X^{*}(T),\G_m)$. A fundamental result (see \cite{ono},\cite{plato}) in the theory of algebraic tori is 

\begin{theorem}\label{t1}
The functor $T\mapsto X^{*}(T)$ sets up an anti-equivalence between the category of $K$-tori split by $L$ and the category of $\Z$-torsion-free finitely generated $\Z[G_{L/K}]$-modules.
\end{theorem}

We now review results regarding the cohomology groups $H^i(K,T):=H^i(G_K,T)$ where $K$ will either be a local field or a global field. In view of Theorem \ref{t1}, it is natural to relate $H^i(K,T)$ to $H^i(K,X^{*}(T))$. First, however, we wish to replace the cohomology of the profinite group $G_K$ by the cohomology of a finite quotient group.

\begin{lemma}\label{l1}
For a $K$-torus $T$ split by a finite Galois extension $L/K$, we have $H^1(K,T)=H^1(L/K,T(L))$.
\end{lemma}

\begin{proof}
The low-degree terms of the Hochschild-Serre spectral sequence $H^r(L/K,H^s(L,T))\Rightarrow H^{r+s}(K,T)$ gives an exact sequence
$$1 \to H^1(L/K,T(L)) \to H^1(K,T) \to H^1(L,T)$$
But $H^1(L,T)=1$ by Hilbert's Theorem 90.
\end{proof}

Now note that the low degree terms of the Hochschild-Serre spectral sequence
$$H^r(L/K,H^s(L,X^{*}(T)))\Rightarrow H^{r+s}(K,X^{*}(T))$$
yields the exact sequence
\begin{align*}
1 &\to H^1(L/K,X^{*}(T)) \to H^1(K,X^{*}(T)) \to H^1(L,X^{*}(T))\\
&\to H^2(L/K,X^{*}(T)) \to \ker\left(H^2(K,X^{*}(T))\to{H^2(L,X^{*}(T))}^G\right)\\
&\to H^1(L/K,H^1(L,X^{*}(T)))
\end{align*}
Since $T$ splits over $L$, we find that $G_L:=\Gal(\overline{L}/L)$ acts trivially on $X^{*}(T)_L\isom\Z^d$ . This implies that 
$$H^1(L,X^{*}(T))\isom H^1(G_L,\Z^d)=\Hom(G_L,\Z^d)$$
and the last group is trivial since $G_L$ is torsion and $\Z^d$ is torsion-free. Thus we have

\begin{lemma}\label{l1a}
For a $K$-torus $T$ split by a finite Galois extension $L/K$ with Galois group $G$, we have 
$H^1(K,X^{*}(T))\isom H^1(L/K,X^{*}(T))$, and the sequence
$$1 \to H^2(L/K,X^{*}(T)) \to H^2(K,X^{*}(T)) \to {H^2(L,X^{*}(T))}^G$$
is exact.
\end{lemma}

It will be convenient for us to use the Tate cohomology groups $\hat{H}^i(L/K,T)$ for all $i$. Recall that for any finite abelian group $M$, $M^{\vee}:=\Hom(M,\Q/\Z)$ denotes the Pontryagin dual of $M$. The first fundamental result that is needed is
the local version of the Nakayama-Tate theorem (\cite{nsw},\cite{plato}).

\begin{theorem}[Nakayama-Tate; Local version]\label{t2}
Given a local field $K$ and a $K$-torus $T$ with splitting field $L$, there is an isomorphism
$$\hat{H}^i(L/K,T)\isom {\hat{H}^{2-i}(L/K,X^{*}(T))}^{\vee}$$
for any $i$.
\end{theorem}

Since $L/K$ is a finite extension and $X^{*}(T)$ is a finitely-generated torsion-free $\Z[G_{L/K}]$-module, $\hat{H}^i(L/K,X^{*}(T))$ is finite for all $i$. It follows from Theorem \ref{t2} that, for all $i$, $\hat{H}^i(L/K,T)$ is finite as well. Furthermore, when combined with Theorem \ref{t1}, this implies that 

\begin{proposition}\label{p1}
For a local field $K$, $H^1(K,T)$ and $H^1(K,X^{*}(T))$ are finite.
\end{proposition}

Let us now consider a global field $K$. Let $\A_K$ denote the adele ring of $K$, $\I_K$ the id\`{e}le group of $K$, and $C_K:=\I_K/K^{\times}$ the id\`{e}le class group of $K$. Let $L/K$ be a finite Galois extension with $G:=\Gal(L/K)$. Then both $\A_L$ and $\I_L$ have the structure of a $G$-module, and
$$\A_L^G=\A_K,\,\,\,\,\I_L^G=\I_K.$$
Let $T$ be a $K$-torus split over $L$. Consider the exact sequence
$$1 \to L^{\times} \to \I_L \to C_L \to 1$$
The functor $\Hom(X^{*}(T),-)$ induces the exact sequence
$$1 \to T(L) \to T(\A_L) \to C_L(T) \to 1$$
of $G$-modules, where $T(\A_L)=\Hom(X^{*}(T),\I_L)$ is the adele ring of $T$ over $L$, and $C_L(T)=T(\A_L)/T(L)$ is the adele class group of $T$. We now state our second fundamental result which relates the global cohomology of $C_L(T)$ to that of $X^{*}(T)$ (\cite{nsw},\cite{plato}).

\begin{theorem}[Nakayama-Tate; Global version]\label{t3}
Given a global field $K$ and a $K$-torus $T$ with splitting field $L$, there is an isomorphism
$$\hat{H}^i(L/K,C_L(T))\isom{\hat{H}^{2-i}(L/K,X^{*}(T))}^{\vee}$$
for any $i$.
\end{theorem}

Now let $v$ be any prime of $K$ and $w$ be a prime of $L$ extending $v$. We denote by $L_w$ and $K_v$ the corresponding completions. The global cohomology of $T(\A_L)$ can be expressed in terms of local cohomology of $T(L_w)$ as follows

\begin{proposition}\label{p2}
For a global field $K$, there is a direct sum decomposition
$$\hat{H}^i(L/K,T(\A_L))\isom\bigoplus_{v}\hat{H}^i(L_w/K_v,T(L_w))$$
\end{proposition}

The exact sequence 
$$1 \to T(L) \to T(\A_L) \to C_L(T) \to 1$$
of $G$-modules induces a corresponding exact sequence of cohomology groups
$$\hat{H}^{i-1}(L/K,C_L(T)) \to \hat{H}^i(L/K,T(L)) \xrightarrow{\varphi_i} \hat{H}^i(L/K,T(\A_L)) \to \hat{H}^i(L/K,C_L(T))$$
The remarks following Theorem \ref{t2} apply in the present situation as well, and we conclude that $\hat{H}^i(L/K,C_L(T))$ is finite for all $i$. It follows that the map $\varphi_i$ above has a finite kernel and cokernel, i.e. $\varphi_i$ is a \emph{quasi-isomorphism} for all $i$. Using Proposition \ref{p2}, we get

\begin{proposition}\label{p3}
For a global field $K$, the group $P^i(L/K,T)$ defined by the exactness of the sequence
$$1 \to P^i(L/K,T) \to \hat{H}^i(L/K,T(L)) \to \prod_{v}\hat{H}^i(L_w/K_v,T(L_w))$$
is finite for all $i$.
\end{proposition}

Note that $P^i(L/K,T)=\ker(\varphi_i)$ is also defined by the exactness of the sequence
$$\hat{H}^{i-1}(L/K,T(\A_L))\to\hat{H}^{i-1}(L/K,C_L(T))\to P^i(L/K,T)\to 1$$
which, using Proposition \ref{p2}, may also be given as
$$\bigoplus_{v}\hat{H}^{i-1}(L_w/K_v,T(L_w))\to\hat{H}^{i-1}(L/K,C_L(T))\to P^i(L/K,T)\to 1$$
Dualizing this sequence and using Theorem \ref{t2} and Theorem \ref{t3}, we obtain the following theorem of Tate

\begin{theorem}[Tate]\label{t4}
There is an exact sequence
$$1 \to P^i(L/K,T)^{\vee} \to \hat{H}^{3-i}(L/K,X^{*}(T)) \to \bigoplus_{v}\hat{H}^{3-i}(L_w/K_v,X^{*}(T))$$
\end{theorem}

We now define the Shafarevich-Tate group of $T$, denoted by $\Sha(T/K)$, by the exactness of the sequence
$$1 \to \Sha(T/K) \to H^1(K,T) \to \prod_{v}H^1(K_v,T)$$
where the product is over all primes of $K$.

\begin{proposition}\label{p4}
$\Sha(T/K)$ is finite for a global field $K$.
\end{proposition}

\begin{proof}
The exact sequences defining $P^1(L/K,T)$ and $\Sha(T/K)$ fit into a commutative diagram
\[
\xymatrix{
1\ar[r] &P^1(L/K,T)\ar[r]\ar[d] &H^1(L/K,T(L))\ar[r]\ar[d] &\prod_{v}H^1(L_w/K_v,T(L_w))\ar[d]\\
1\ar[r] &\Sha(T/K)\ar[r] &H^1(K,T)\ar[r] &\prod_{v}H^1(K_v,T)}
\]
where the middle and right vertical arrows are inflation maps, inducing the left vertical map. However, $T$ is split over $L$ and hence, over each $L_w$. By Lemma \ref{l1},
both the middle and the right vertical arrows are isomorphisms. Consequently, $\Sha(T/K)\isom P^1(L/K,T)$ and the latter group is finite by Proposition \ref{p3}.
\end{proof}

Since $\Sha(T/K)\isom P^1(L/K,T)$ and since any finite group is isomorphic to its Pontryagin dual, Theorem \ref{t4} implies that

\begin{corollary}\label{c1}
There is an exact sequence
$$1 \to \Sha(T/K) \to H^2(L/K,X^{*}(T)) \to \bigoplus_{v}H^2(L_w/K_v,X^{*}(T))$$
\end{corollary}

We now state another useful characterization of $\Sha(T/K)$.

\begin{proposition}\label{p5}
There is an isomorphism
$$\Sha(T/K)\isom C_L(T)^G/C_K(T)$$
In particular, $\Sha(T/K)$ measures the failure of Galois descent for $C_L(T)$.
\end{proposition}

\begin{proof}
The exact sequence
$$1 \to T(L) \to T(\A_L) \to C_L(T) \to 1$$
of $G$-modules yields the exact sequence
$$1 \to T(K) \to T(\A_K) \to C_L(T)^G \to P^1(L/K,T) \to 1$$
The requisite isomorphism now follows immediately since we have that $C_K(T)\isom T(\A_K)/T(K)$ and $\Sha(T/K)\isom P^1(L/K,T)$.
\end{proof}

\section{Class field theory}
Let $K$ be a global field and $K_v$ the completion of $K$ at a prime $v$. There is a \emph{local invariant map} $\inv_v : \Br(K_v)\to\Q/\Z$ for each $v$ such that $\inv_v$ is an isomorphism. A fundamental result of global class field theory is

\begin{theorem}\label{tcft}
There is an exact sequence
$$1 \to \Br(K) \to \bigoplus_{v}\Br(K_v) \xrightarrow{\sum_v\inv_v} \Q/\Z \to 1$$
where $\sum_{v}\inv_v$ is the sum of the local invariant maps.
\end{theorem}

The exact sequence in Theorem \ref{tcft} is known as the \emph{fundamental exact sequence} of global class field theory for $K$.
Let $L/K$ be a finite Galois extension with $G_{L/K}:=\Gal(L/K)$. Let $w$ be a prime of $L$ extending $v$, and $L_w$ the corresponding completion. 

\begin{lemma}\label{l3}
Suppose that $L/K$ is a finite cyclic extension. There is an exact sequence
$$1 \to \Br(L/K) \to \bigoplus_{v}\Br(L_w/K_v) \to C_K/\Nm{C_L} \to 1$$
where $\Nm$ is the norm map associated to $L/K$.
\end{lemma}

\begin{proof}
By mapping the fundamental exact sequence for $K$ to the fundamental exact sequence for $L$ and applying the snake lemma to the resulting diagram, we obtain the top row of the diagram
\[
\xymatrix{
1\ar[r] &\Br(L/K)\ar[r]\ar[d] &\bigoplus_{v}\Br(L_w/K_v)\ar[r]\ar[d] &\Q/\Z\\
1\ar[r] &\hat{H}^0(L/K,L^{\times})\ar[r] &\bigoplus_{v}\hat{H}^0(L_w/K_v,L_w^{\times})\ar[r] &\hat{H}^0(L/K,C_L)\ar[r] &1}
\]
The bottom row is obtained from the Tate cohomology sequence of 
$$1 \to L^{\times} \to \I_L \to C_L \to 1$$
and from the fact that 
$$\hat{H}^i(L/K,\I_L)\isom\bigoplus_{v}\hat{H}^i(L_w/K_v,L_w^{\times})$$
together with the fact that $\hat{H}^{-1}(L/K,C_L)=1=\hat{H}^1(L/K,L^{\times})$. Since $\Br(L/K)=H^2(L/K,L^{\times})$ and likewise for $\Br(L_w/K_v)$, periodicity of the Tate cohomology of finite cyclic groups imply that the vertical arrows in the above diagram are isomorphisms. Hence, the image of the map $\bigoplus_{v}\Br(L_w/K_v)\to\Q/\Z$ is $\hat{H}^0(L/K,C_L)=C_K/\Nm{C_L}$.
\end{proof}

\section{Torsors under algebraic tori}

Consider again a $d$-dimensional torus $T$ defined over any perfect field $K$. Let $X$ be any $K$-torsor under $T$. In particular, $X$ is a smooth, geometrically integral $K$-variety. The \'{e}tale cohomology group $H^2(X,\G_m)$ is called the \emph{cohomological Brauer group} of $X$ \cite{ec} and denoted by $\Br(X)$, and $\Pic(X)=H^1(X,\G_m)$ is the \emph{Picard group} of $X$. The kernel of the map $\Br(X)\to\Br(\overline{X})$ is denoted by $\Br_1(X)$ and called the \emph{algebraic Brauer group} of $X$. The low degree terms of the Hochschild-Serre spectral sequence
$$H^p(G_K, H^q(\overline{X}, \G_m))\Longrightarrow H^{p+q}(X,\G_m)$$
yields the exact sequence
\begin{align*}
1 &\to H^1(K,\overline{K}[X]^{*}) \to H^1(X,\G_m) \to {H^1(\overline{X},\G_m)}^{G_K}\\
&\to H^2(K,\overline{K}[X]^{*}) \to \ker(H^2(X,\G_m)\to {H^2(\overline{X},\G_m)}^{G_K})\\ 
&\to H^1(K,H^1(\overline{X},\G_m))
\end{align*}
Here, $\overline{K}{[X]}^{*}$ is the group of invertible functions on $X$. Noting that $H^1(\overline{X},\G_m)=\Pic(\overline{X})=\Pic(\overline{T})=1$, we obtain isomorphisms
$$\Pic(X)\isom H^1(K, \overline{K}{[X]}^{*})\;\textrm{and}\;\Br_1(X)\isom H^2(K, \overline{K}{[X]}^{*})$$
By a lemma of Rosenlicht (see \cite{sko}), there is an exact sequence
$$1 \to {\overline{K}}^{\times} \to \overline{K}{[X]}^{*} \to X^{*}(T) \to 1$$
of $G_K$-modules. The corresponding long exact sequence of cohomology, combined with Hilbert's Theorem 90 yields the exact sequence
\begin{align*}
1 &\to H^1(K, \overline{K}{[X]}^{*}) \to H^1(K, X^{*}(T)) \to \Br(K)\\
&\to H^2(K, \overline{K}{[X]}^{*}) \to H^2(K, X^{*}(T)) \to 1
\end{align*}
where the $1$ on the right follows from the fact that $H^3(K, {\overline{K}}^{\times})=1$. Combining this with the isomorphisms above, we obtain

\begin{lemma}\label{l4}
There is an exact sequence
$$1 \to \Pic(X) \to H^1(K,X^{*}(T)) \to \Br(K) \to \Br_1(X) \to H^2(K,X^{*}(T)) \to 1$$
\end{lemma}

Suppose now that $X$ has a $K$-rational point. Then the retraction $\Spec{K}\to X\to \Spec{K}$ shows that the sequence
$$1 \to H^n(K,{\overline{K}}^{\times}) \to H^n(K,\overline{K}{[X]}^{*}) \to H^n(K,X^{*}(T)) \to 1$$
is exact and split for all $n\geq{1}$. It then follows from the discussion above that we have

\begin{lemma}\label{l5}
Suppose that $X(K)\neq\emptyset$. Then $\Pic(X)\isom H^1(K,X^{*}(T))$, and there is an exact sequence
$$1 \to \Br(K) \to \Br_1(X) \to H^2(K,X^{*}(T)) \to 1$$
\end{lemma}

\section{Main Results}
We now prove the first main theorem of this paper.

\begin{theorem}\label{t6}
Suppose that $L/K$ is a finite, cyclic extension of perfect fields with Galois group $G$. Let $T$ be a $d$-dimensional $K$-torus that is split over $L$, and let $X$ be a $K$-torsor under $T$. Let $\Br_1(X_{L/K})$ denote the kernel of the map $\Br_1(X)\to\Br_1(X_L)^G$. Then there is an exact sequence
$$1\to H \to \Br(L/K) \to \Br_1(X_{L/K}) \to H^2(L/K,X^{*}(T)) \to 1$$
where $H$ is a finite group of order $\displaystyle\frac{[H^1(G, X^{*}(T))]}{[\Pic(X)]}$.
\end{theorem}

\begin{proof}
Consider the commutative diagram
\[
\xymatrix{
&\Br(K)\ar[r]\ar[d] &\Br_1(X)\ar[r]\ar[d] &H^2(K,X^{*}(T))\ar[r]\ar[d] &1\\
1\ar[r] &{\Br(L)}^G\ar[r] &{\Br_1(X_L)}^G\ar[r] &{H^2(L,X^{*}(T))}^G}
\]
Here the top row is part of the exact sequence in Lemma \ref{l4}. The bottom row also follows from the same lemma over $L$ by taking $G$-invariants and by noting that $H^1(L, X^{*}(T))=1$. The vertical arrows are restriction maps on cohomology induced by the inclusion $K\subseteq{L}$. The Hochschild-Serre spectral sequence gives the exact sequence
$$1 \to \Br(L/K) \to \Br(K) \to {\Br(L)}^G \to H^3(L/K,L^{\times})$$
while the periodicity of the cohomology of cyclic groups give
$$H^3(L/K,L^{\times})=H^1(L/K,L^{\times})=1$$
Thus the leftmost vertical map in the diagram is surjective, and has kernel $\Br(L/K)$. The middle vertical map has kernel $\Br_1(X_{L/K})$ and the rightmost vertical map has kernel $H^2(L/K,X^{*}(T))$ by Lemma \ref{l1a}. The extended snake lemma produces the exact sequence
\begin{align*}
1 \to &\ker(\Br(K)\to\Br_1(X)) \to \Br(L/K) \to \Br_1(X_{L/K}) \\
&\to H^2(L/K,X^{*}(T)) \to 1
\end{align*}
Denoting $\ker(\Br(K)\to\Br_1(X))$ by $H$, we get the desired exact sequence. By Lemma \ref{l1a} and Lemma \ref{l4}, we find that there is an exact sequence
$$1 \to \Pic(X) \to H^1(L/K, X^{*}(T)) \to H \to 1$$
By Theorem \ref{t3}, the group $H^1(L/K, X^{*}(T))$ is finite and hence, so is $\Pic(X)$.
\end{proof}

Denoting by $\Br_0(X_{L/K})$ the image of the map $\Br(L/K)\to\Br_1(X_{L/K})$, we thus get an isomorphism
$$\Br_1(X_{L/K})/\Br_0(X_{L/K})\isom H^2(L/K,X^{*}(T))$$
In particular, the quotient group $\Br_1(X_{L/K})/\Br_0(X_{L/K})$ is finite. We now note that in the proof of Theorem \ref{t6}, the cyclicity of $L/K$ is used only in establishing the triviality of $H^3(L/K,L^{\times})$. However, this also follows without the cyclicity assumption when we have a finite, Galois extension of local fields, by \cite[Cor 7.2.2]{nsw}. The next corollary then follows immediately from Theorem \ref{t6} and Theorem \ref{t2}.

\begin{corollary}
Let $L/K$ be a finite, Galois extension of local fields. Then there is a canonical perfect pairing of finite abelian groups
$$\Br_1(X_{L/K})/\Br_0(X_{L/K})\times T(K)/\Nm T(L)\to\Q/\Z$$
\end{corollary}

For the rest of this section, we fix a finite, cyclic extension $L/K$ of global fields. The next corollary follows from Theorem \ref{t6} and Theorem \ref{t3}.

\begin{corollary}\label{c2}
There is a canonical perfect pairing of finite abelian groups
$$\Br_1(X_{L/K})/\Br_0(X_{L/K})\times C_L(T)^G/\Nm C_L(T)\to\Q/\Z$$
\end{corollary}

\begin{remark}
When $L/K$ is a finite Galois extension, it is clear that we have an injection 
$\Br_1(X_{L/K})/\Br_0(X_{L/K})\hra H^2(L/K,X^{*}(T))$ such that the corresponding quotient
is a subgroup of $H^3(L/K,L^{\times})$.
\end{remark}

\begin{corollary}\label{c3}
Under the conditions of Theorem \ref{t6}, suppose further that $X(K)\neq\emptyset$. Then we have
$$\displaystyle\frac{[\Br_1(T_{L/K})/\Br_0(T_{L/K})]}{[\Pic(T)]}=h(G,C_L(T))$$
where $h(G,C_L(T))$ is the Herbrand quotient of $C_L(T)$.
\end{corollary}

\begin{proof}
Note that $X(K)\neq\emptyset$ implies that $X\isom T$ over $K$. The result now follows from Lemma \ref{l5}, Lemma \ref{l1a}, and Theorem \ref{t3}.
\end{proof}

\begin{remark}
The \emph{Tamagawa number} of $T$, denoted by $\tau(T)$, is defined as the volume of a certain homogeneous space associated to $T(\A_K)$ with respect to the \emph{Tamagawa measure} \cite{weil}. Using the remarkable result of Ono \cite{ono1}
$$\tau(T)=\displaystyle\frac{[H^1(K,X^{*}(T))]}{[\Sha(T/K)]}$$
we obtain from Corollary \ref{c3} that
$$\displaystyle\frac{[\Br_1(T_{L/K})/\Br_0(T_{L/K})]}{[\Sha(T/K)]}=\tau(T)\,h(G,C_L(T))$$
\end{remark}

\begin{proposition}
Suppose that $T$ is a one dimensional torus and $X$ is a $K$-torsor under $T$. Assuming the conditions of Theorem \ref{t6}, we have
$$[\Br_1(X_{L/K})/\Br_0(X_{L/K})]=[L:K]$$
\end{proposition}

\begin{proof}
Since $T$ has dimension one, we have $C_L(T)=C_L$. Corollary \ref{c2} then implies that
$$[\Br_1(X_{L/K})/\Br_0(X_{L/K})]=[\hat{H}^0(G,C_L)]=[L:K]$$
where the last equality follows from \cite[Thm 8.1.1]{nsw}
\end{proof}

We now prove the second main theorem of this paper. Let $v$ be any prime of $K$ and $w$ be that of $L$ dividing $v$. We denote the corresponding completions by $K_v$ and $L_w$. Let $\Br_1'(X_{L/K})$ and $\Cok$ be defined by the exactness of the sequence
$$1 \to \Br_1'(X_{L/K}) \to \Br_1(X_{L/K}) \to \bigoplus_{v}\Br_1(X_{L_w/K_v}) \to \Cok \to 1$$
where the sum is over all primes $v$ of $K$.

\begin{theorem}\label{t7}
Assume the hypothesis of Theorem \ref{t6}. Suppose further that $X(K_v)\neq\emptyset$ for every prime $v$. Then there is an exact sequence
$$1 \to \Br_1'(X_{L/K}) \to \Sha(T/K) \to C_K/ \Nm{C_L}\to \Cok \to C \to 1$$
where $C$ is a finite group of order $[T(\A_K)\cap\Nm(C_L(T)):\Nm(T(\A_L))]$. 
\end{theorem}

\begin{proof}
Consider the commutative diagram
\[
\xymatrix{
&\Br(L/K)\ar[r]\ar^{\alpha_1}[d] &\Br_1(X_{L/K})\ar[r]\ar^{\alpha_2}[d] &H^2(L/K,X^{*}(T))\ar[r]\ar^{\alpha_3}[d] &1\\
1\ar[r] &\bigoplus_{v}\Br(L_w/K_v)\ar[r] &\bigoplus_{v}\Br_1(X_{L_w/K_v})\ar[r] &\bigoplus_{v}H^2(L_w/K_v,X^{*}(T))\ar[r] &1}
\]
Both rows follow from the exact sequence in Theorem \ref{t6}. The $1$ at the left of the bottom row follows from Lemma \ref{l5}. Snake lemma then gives an exact sequence
$$\ker{\alpha_1} \to \ker{\alpha_2} \to \ker{\alpha_3} \xrightarrow{\delta} \coker{\alpha_1} \to \coker{\alpha_2} \to \coker{\alpha_3} \to 1$$
By Lemma \ref{l3}, we have $\ker{\alpha_1}=1$ and $\coker{\alpha_1}\isom C_K/\Nm{C_L}$. Clearly, $\ker{\alpha_2}=\Br_1'(X_{L/K})$ and $\coker{\alpha_2}=\Cok$. On the other hand, we have $\ker{\alpha_3}\isom{\Sha(T/K)}^{\vee}\isom\Sha(T/K)$ (since $\Sha(T/K)$ is finite) by Corollary \ref{c1}. The exact sequence above now is
$$1 \to \Br_1'(X_{L/K}) \to \Sha(T/K) \xrightarrow{\delta} C_K/\Nm{C_L}\to\Cok\to\coker{\alpha_3}\to 1$$
To finish the proof, we note that
\begin{align*}
\coker{\alpha_3}&=\coker\left(H^2(L/K, X^{*}(T))\to\bigoplus_{v}H^2(L_w/K_v,X^{*}(T))\right)\\
&\isom\coker\left({\hat{H}^0(L/K,C_L(T))}^{\vee}\to{\hat{H}^0(L/K,T(\A_L))}^{\vee}\right)\\
&\isom\left(\ker\left(\hat{H}^0(L/K,T(\A_L))\to\hat{H}^0(L/K,C_L(T))\right)\right)^{\vee}\\
&\isom\left(\ker\left(T(\A_K)/\Nm{T(\A_L)}\to{C_L(T)}^G/\Nm{C_L(T)}\right)\right)^{\vee}\\
&\isom\left({T(\A_K)\cap\Nm{C_L(T)}}/\Nm{T(\A_L)}\right)^{\vee}
\end{align*}

\end{proof}

The following corollary is immediate.

\begin{corollary}
We have
$$\frac{[\Br'_1(X_{L/K})]}{[\Cok]}=\frac{[\Sha(T/K)]}{[L:K][T(\A_K)\cap\Nm(C_L(T)):\Nm(T(\A_L))]}$$
\end{corollary}

By Theorem \ref{t7}, we can identify $\Br'_1(X_{L/K})$ as the kernel of the map $\Sha(T/K)\to C_K/\Nm{C_L}$. Hence, by Proposition \ref{p5}, we get

\begin{proposition}
There is an isomorphism
$$\Br'_1(X_{L/K})\isom{C_L(T)}^{G}\cap\Nm(C_L)/C_K(T)$$
\end{proposition}

\newpage
\newcommand{\etalchar}[1]{$^{#1}$}

\end{document}